\documentclass{article}
\usepackage{graphicx} 

\usepackage{a4wide}
\usepackage[utf8]{inputenc}
\usepackage[a4paper, margin=2.3cm]{geometry}
\usepackage{amsmath,amssymb,amsthm}
\usepackage{color}
\usepackage{xcolor}
\usepackage{parskip}
\usepackage{hyperref}
\usepackage{parskip}
\usepackage{setspace}
\usepackage{enumitem}
\usepackage{comment}
\usepackage{float} 

\usepackage{geometry}
\usepackage{booktabs}
\newtheorem{theorem}{Theorem}
\newtheorem{proposition}[theorem]{Proposition}
\newtheorem{lemma}[theorem]{Lemma}
\newtheorem{corollary}[theorem]{Corollary}
\newtheorem{remark}[theorem]{Remark}

\newcommand{\F}{\mathbb{F}}

\newcommand{\abs}[1]{\left|#1\right|}

\title{Parabolic distance in $\F_q^2$: a sharp exponent and new results}
\author{Dao Nguyen Van Anh\thanks{The Dewey Schools, Hanoi. ~~~~~~~~~~~~~~~~~~~~~~~~~~~~~~~~~~~~~~~~~~ Email: {\tt anh.daonguyenvan@thedeweyschools.edu.vn}}  \and Steven Senger\thanks{Department of Mathematics, Missouri State University. ~~~~~ Email: {\tt StevenSenger@MissouriState.edu}} \and Dung The Tran \thanks{VNU University of Science, Hanoi, Vietnam. ~~~~~~~~~~~~~~~~~~~ Email: {\tt tranthedung56@gmail.com}} \and Le Anh Vinh\thanks{Vietnam Institute of Educational Sciences. ~~~~~~~~~~~~~~~~~~~~~~ Email: {\tt vinhle@vnies.edu.vn}}} 
\date{}

\begin{document}
\maketitle

\begin{abstract}
We study the parabolic variant of the Erd\H os--Falconer distance problem in finite fields. That is, if $q$ is odd, we seek size thresholds beyond which any subset $E\subset \mathbb F_q^2$ will determine many distinct parabolic distances. This problem has a rich history because the parabolic distance functional shares many properties with the standard distance functional, but exhibits many distinct behaviors. Here we begin with rather standard Fourier analytic arguments, but diverge into additive combinatorics to handle the central obstructions. We provide a suite of positive results and corresponding sharpness examples.
\end{abstract}




\tableofcontents

\section{Introduction}
Let $q$ be an odd prime power, and let $\F_q$ denote the finite field with $q$ elements. The distance between two points $\mathbf{x}=(x_1, x_2)$ and $\mathbf{y}=(y_1, y_2)$ in $\mathbb{F}_q^2$ is defined by 
\[\|\mathbf{x}-\mathbf{y}\|=(x_1-y_1)^2+(x_2-y_2)^2.\]
For $E\subset \mathbb{F}_q^2$, let $\Delta(E)$ be the set of distinct distances determined by pairs of points in $E$, i.e. 
\[\Delta(E):=\{\|\mathbf{x}-\mathbf{y}\|\colon \mathbf{x}, \mathbf{y}\in E\}.\]

The well-known Erd\H{o}s--Falconer distance conjecture in $\mathbb{F}_q^2$ states that if $|E|\gg q$ then $|\Delta(E)|\gg q$. Throughout the paper, we adopt the notation $A\gg B$ to mean $A \ge C B$ for some constant $C>0$ independent of $q$ (or $p$), and $A\sim B$ to denote $C_1 A \ge B \ge C_2 A$ for constants $C_1, C_2>0$ independent of $q$ (or $p$).

In \cite{IR07}, Iosevich and Rudnev used a Fourier-analytic framework to prove that if $|E|\ge 4q^{ \frac{3}{2}}$, then $\Delta(E)=\mathbb{F}_q$. The exponent $\frac{3}{2}$ was subsequently improved to $\frac{4}{3}$ by Chapman, Erdo\u{g}an, Hart, Iosevich, and Koh \cite{CEHIK12} using tools from restriction/extension theory associated to circles. In a 2022-breakthrough paper, Murphy, Petridis, Pham, Rudnev, and Stevens \cite{MPPRS22} further reduced this exponent to $\frac{5}{4}$ over prime fields by combining algebraic methods with incidence geometry.

In a recent paper \cite{phamYoo23}, Pham and Yoo discovered several surprising applications of the distance problems to intersection patterns and incidence geometry. They also proved a rotational version of the Erd\H{o}s--Falconer distance conjecture by showing that for any two sets $E,F\subset \mathbb{F}_p^2$, if $|E|,|F|\gg p$, then for almost every rotation $g\in SO_2(\mathbb{F}_p)$, one has
\[
|\Delta(E,gF)|\gg p,
\]
where
\[
\Delta(E,gF):=\{\|\mathbf{x}-g\mathbf{y}\|:\mathbf{x}\in E,\ \mathbf{y}\in F\}.
\]
Equivalently, for each rotation $g\in SO_2(\mathbb{F}_p)$, one may consider the distance map
\[
f_g(\mathbf{x},\mathbf{y}):=\|\mathbf{x}-g\mathbf{y}\|,
\qquad (\mathbf{x},\mathbf{y})\in \mathbb{F}_p^2\times\mathbb{F}_p^2,
\]
so that $\Delta(E,gF)$ is precisely the image set of $f_g$ on $E\times F$. Thus, the result of Pham and Yoo shows that for sufficiently large sets $E$ and $F$, most members of this family take many distinct values on $E\times F$.

In addition to the progression of study of the usual distance functional, it is natural to ask whether a similar phenomenon persists for other functionals. Indeed, the study of other such related problems has had its own parallel branch of exploration. In \cite{KS12}, Koh and Shen proved a suite of results about a wide family of functionals similar the standard distance functional given here, except coming from more general types of polynomials. Iosevich and Hart explored the dot product in \cite{IH08}, and this lead to improved results on the sum-product phenomena in finite fields. Another important example is the direction functional, studied by Iosevich, Morgan, and Pakianathan in \cite{IMP11}, the exploration of which lead to the affirmative answer of the Fuglede conjecture for finite fields in two dimensions, due to Iosevich, Mayeli, and Pakianathan in \cite{IMP17}. This is such a wide area of study that to list all relevant work here would be infeasible, but the interested reader can look at \cite{CL24, kohkoh, KPV2021, L19, LMP20, phamboxing24} and the references contained therein for a more complete picture.

Another particularly compelling example that has seen interest over the years is the paraboloid distance function. In this setting, the exponent of $\frac{3}{2}$ with the full-distance conclusion are relatively straightforward to obtain (we give a complete statement and proof in Section \ref{appendix-section} below), whereas $\frac{4}{3}$ is the first genuinely nontrivial benchmark.\footnote{We thank Alex Iosevich for raising this question in several conversations with the first and fourth listed authors.} This benchmark is especially natural in view of the classical distance problem, where the $\frac{4}{3}$-threshold is tied to restriction/extension estimates associated to circles. Although the corresponding restriction theory for spheres is notoriously difficult, substantially more is known for paraboloids; see \cite{IK10} due to Iosevich and Koh. This suggests that the paraboloid distance problem may be more accessible at the $\frac{4}{3}$ level, and it is this question that motivates the present paper.

For $\mathbf{x}=(x_1,x_2)$ and $\mathbf{y}=(y_1,y_2)$ in $\F_q^2$, define the \emph{parabolic distance} by
\begin{equation}\label{eq:def-parabolic-distance}
\|\mathbf{x}-\mathbf{y}\|_{P}:=(x_2-y_2)+(x_1-y_1)^2,
\end{equation}
and, for $E\subset \F_q^2$, let
\begin{equation*}
\Delta_P(E):=\{\|\mathbf{x}-\mathbf{y}\|_{P}: \mathbf{x}, \mathbf{y}\in E\}.
\end{equation*}

In contrast to the classical distance function, we first construct examples showing that the exponent $\frac{3}{2}$ is sharp for the parabolic distance function.

\begin{theorem}\label{thm2}
The exponent $\frac{3}{2}$ is optimal. More precisely:
\begin{enumerate}
    \item Over prime fields, for every $\varepsilon\in(0,1)$, there exists a set $E\subset \F_p^2$ such that
    \[
    |E|=p^{\frac32-\varepsilon}
    \qquad\text{and}\qquad
    |\Delta_P(E)|=o(p).
    \]
    \item Over arbitrary finite fields, for every rational number $\varepsilon\in(0,\tfrac12)$, there exists a set $E\subset \F_q^2$ such that
    \[
    |E|=q^{\frac32-\varepsilon}
    \qquad\text{and}\qquad
    |\Delta_P(E)|=o(q).
    \]
\end{enumerate}
\end{theorem}

Theorem \ref{thm2} shows that the exponent $\frac{3}{2}$ is sharp in full generality. It is therefore natural to ask which configurations of $E$ can make $|\Delta_P(E)|$ small. The main obstruction is concentration on vertical lines: if many points of $E$ share the same first coordinate, then the parabolic distance set may be significantly smaller than expected.

To quantify this obstruction, we introduce
\begin{equation}\label{eq:def-K}
K_E:=\max_{u\in \F_q}\abs{E\cap (\{u\}\times \F_q)},
\end{equation}
the size of the largest vertical slice of $E$. For ease of exposition, we suppress the subscript and simply write $K$ in place of $K_E,$ when context is clear. Our main result gives a quantitative lower bound for $\abs{\Delta_P(E)}$ in terms of $\abs{E}$ and $K_E$.

\begin{theorem}\label{thm:main}
Let $E\subset \F_q^2$, and let $K_E$ be defined as in \eqref{eq:def-K}. Then
\begin{equation}\label{eq:main-lower-bound}
\abs{\Delta_P(E)}
\ge
\frac{q}{1+\dfrac{q^2K_E}{|E|^2}}.
\end{equation}
In particular, if $|E|\ge Cq\sqrt{K_E}$ for a sufficiently large absolute constant $C$, then $\abs{\Delta_P(E)}\ge cq$ for some absolute constant $c>0$.
\end{theorem}

An immediate consequence is the following.
\begin{corollary}\label{cor:alpha}
Suppose $K_E\ll |E|^\alpha$ for some $0\le \alpha<1$. If $|E|\gg q^{\frac{2}{2-\alpha}}$, then $|\Delta_P(E)|\gg q$.
\end{corollary}

\noindent\textbf{Organization of the paper.} The rest of this paper is organized as follows. In Section~\ref{sec:preliminaries}, we lay out our tools. To make the exposition fairly self-contained, we include a short proof of the threshold $\frac{3}{2}$ in Section~\ref{appendix-section}. We prove Theorem \ref{thm:main} in Section~\ref{sec:proof-main}. We demonstrate the general sharpness of Theorem \ref{startingpoint} by proving Theorem \ref{thm2} in Section~\ref{sec:prime-counterexample}. We then prove some natural generalizations of Theorem \ref{thm:main} in Section~\ref{extension-section}, followed by demonstrating the sharpness of Theorem \ref{thm:main} in Section~\ref{section6}.

\section{Preliminaries}\label{sec:preliminaries}

Throughout, $q$ is an odd prime power, and $\chi:\F_q\to\mathbb C^\times$
denotes a fixed nontrivial additive character. We use the orthogonality relation
\begin{equation}\label{eq:orthogonality}
\sum_{t\in \F_q}\chi(\alpha t)=
\begin{cases}
q, & \alpha=0,\\
0, & \alpha\ne 0.
\end{cases}
\end{equation}
For a function $h:\F_q\to\mathbb C$, define its unnormalized Fourier transform by
\[
\widehat{h}(\xi):=\sum_{y\in \F_q}h(y)\chi(\xi y),
\qquad \xi\in \F_q.
\]

For every $h:\F_q\to \mathbb{C}$, the Plancherel tells us that
\begin{equation}\label{lem:plancherel}
\sum_{\xi\in \F_q}\abs{\widehat{h}(\xi)}^2
=
q\sum_{y\in \F_q}\abs{h(y)}^2.
\end{equation}

For finite sets $P,Q\subset \F_q$, define their additive energy by
\[
E_+(P,Q):=
\abs{\{(x,x',y,y')\in P^2\times Q^2:\ x+y=x'+y'\}}.
\]
We record the following trivial bound on the additive energy.
\begin{lemma}\label{lem:energy-trivial}
For finite sets $P,Q\subset \F_q$,
\begin{equation*}
E_+(P,Q)
\le
|P|\,|Q|\,\min\{|P|,|Q|\}
\le
|P|^{\frac{3}{2}}|Q|^{\frac{3}{2}}.
\end{equation*}
\end{lemma}

\begin{proof}
Let $r_{P+Q}(w):=\abs{\{(x,y)\in P\times Q: x+y=w\}}$. Then, by Cauchy--Schwarz,
\[
E_+(P,Q)=\sum_{w\in \F_q}r_{P+Q}(w)^2
\le
\Bigl(\max_{w\in \F_q}r_{P+Q}(w)\Bigr)\sum_{w\in \F_q}r_{P+Q}(w)
\le
\min\{|P|,|Q|\}\, |P| \, |Q|,
\]
which proves the first inequality. The second follows from the elementary bound
$\min\{u,v\}\le(uv)^{\frac{1}{2}}$.
\end{proof}

\section{The exponent of $\frac{3}{2}$}\label{appendix-section}
Here, we formally state and prove the result giving the exponent of $\frac{3}{2}.$ This has been known for quite some time, but we include this here to make the exposition more self-contained.
\begin{theorem}\label{startingpoint}
    Let $E\subset \mathbb{F}_q^2$. Assume that $|E|\gg q^{\frac{3}{2}}$, then $\Delta_P(E)=\mathbb{F}_q$.
\end{theorem}
We give a short proof of Theorem \ref{startingpoint} as a corollary of the following result due to the fourth listed author.
\begin{theorem}[Vinh {\cite{LAV11}}]\label{thm:Vinh-point-line}
Let $P$ be a set of points and let $L$ be a set of lines in $\F_q^2$. Denote by
\[
I(P,L):=\bigl|\{(p,\ell)\in P\times L:\ p\in \ell\}\bigr|
\]
the number of incidences between $P$ and $L$. Then
\[
\left|I(P,L)-\frac{|P| |L|}{q}\right|
\le q^{\frac{1}{2}}|P|^{\frac{1}{2}}|L|^{\frac{1}{2}}.
\]
\end{theorem}

We now prove Theorem \ref{startingpoint}.
\begin{proof}[Proof of Theorem \ref{startingpoint}]
For each $t\in \F_q$, let
\[
\nu(t):=\bigl|\{(\mathbf{x}, \mathbf{y})\in E\times E:\ \|\mathbf{x}-\mathbf{y}\|_{P}=t\}\bigr|.
\]
It suffices to show that $\nu(t)>0$ for every $t\in \F_q$.

Fix $t\in \F_q$. Write $\mathbf{x}=(x_1,x_2), \mathbf{y}=(y_1,y_2)$. 
By definition $\|\mathbf{x}-\mathbf{y}\|_{P}=(x_2-y_2)+(x_1-y_1)^2$.
Therefore, the relation $\|\mathbf{x}-\mathbf{y}\|_{P}=t$ is equivalent to
\[
x_2-y_2+(x_1-y_1)^2=t.
\]
Expanding the square and rearranging terms, we obtain
\[
x_2+x_1^2=2y_1x_1+(y_2-y_1^2+t).
\]
This suggests the following incidence model. Define the point set
\[
P_t:=\{(x_1,x_2+x_1^2):\ \mathbf{x}=(x_1,x_2)\in E\}\subset \F_q^2,
\]
and define the line set
\[
L_t:=\bigl\{\ell_{\mathbf{y},t}:\ \mathbf{y}=(y_1,y_2)\in E\bigr\},
\]
where
\[
\ell_{\mathbf{y},t}:=\{(u,v)\in \F_q^2:\ v=2y_1u+(y_2-y_1^2+t)\}.
\]
Then for $\mathbf{x}\in E$ and $\mathbf{y}\in E$, the condition $\|\mathbf{x}-\mathbf{y}\|_{P}=t$ holds if and only if $(x_1,x_2+x_1^2)\in \ell_{\mathbf{y},t}$.
Consequently,
\[
\nu(t)=I(P_t,L_t).
\]
Since the map $(x_1,x_2)\mapsto (x_1,x_2+x_1^2)$
is injective, we have $|P_t|=|E|$.
Likewise, for fixed $t$, the map $(y_1,y_2)\mapsto \ell_{\mathbf{y},t}$
is injective, because the slope $2y_1$ determines $y_1$, and then the intercept
$y_2-y_1^2+t$ determines $y_2$. 
Hence
$|L_t|=|E|$.

Applying Theorem \ref{thm:Vinh-point-line}, we conclude that
\[
\left|\nu(t)-\frac{|E|^2}{q}\right|
=
\left|I(P_t,L_t)-\frac{|P_t| |L_t|}{q}\right|
\le q^{\frac{1}{2}}|P_t|^{\frac{1}{2}}|L_t|^{\frac{1}{2}}
= q^{\frac{1}{2}}|E|.
\]
Thus
\[
\nu(t)\ge \frac{|E|^2}{q}-q^{\frac{1}{2}}|E|.
\]
In particular, if $|E|>q^{\frac{3}{2}}$, then $\nu(t)>0$ for every $t\in \F_q$. Therefore, every $t\in \F_q$ is attained as a parabolic distance determined by $E$, and so
\[
\Delta_P(E)=\F_q.
\]
This completes the proof.
\end{proof}

\section{Proof of Theorem \ref{thm2}}\label{sec:prime-counterexample}
Theorem \ref{thm2} follows from the following two propositions. The proof of the first one is similar to a construction due to the first listed author and Iosevich, from \cite{IS14}, though the goal there was to show that one parabolic distance repeated many times there was no analysis of distinct parabolic distances. That result was directly motivated by a construction due to Valtr in \cite{V05}.

 \begin{proposition}\label{prop:prime-counterexample}
Let $p$ be a sufficiently large odd prime and let $\varepsilon \in (0,1)$.
Then there exists a set $E\subset \F_p^2$ such that
\[
|E| \sim p^{\frac{3}{2}-\varepsilon}
\qquad\text{and}\qquad
|\Delta_P(E)|=o(p).
\]
\end{proposition}

\begin{proof}
Let $p$ be an odd prime and fix $\varepsilon\in(0,1)$.
Define
\[
M:=\big\lfloor p^{\frac12-\frac{\varepsilon}{2}}\big\rfloor,
\quad
N:=\big\lfloor p^{1-\frac{\varepsilon}{2}}\big\rfloor,
\quad
A:=\{0,1,\dots,M-1\}\subset \F_p,
\quad
B:=\{0,1,\dots,N-1\}\subset \F_p,
\]
and $E:=A\times B\subset \F_p^2$. Below we show that for all sufficiently large primes $p$ (depending on $\varepsilon)$,
\begin{align*}
    |E| = MN \sim p^{\frac{3}{2}-\varepsilon}
\qquad\text{and}\qquad
|\Delta_P(E)| = (M-1)^2+2N-1 \sim p^{1-\frac{\varepsilon}{2}} = o(p).
\end{align*}

Write points as $\mathbf{x}=(a,b)$ and $\mathbf{y}=(a',b')$ with $a,a'\in A$ and $b,b'\in B$. Then
\[
\|\mathbf{x}-\mathbf{y}\|_P=(b-b')+(a-a')^2.
\]
Let
\[
D:=B-B=\{-(N-1),\dots,N-1\}\subset \F_p,
\qquad |D|=2N-1,
\]
and
\[
Q:=\{(a-a')^2:a,a'\in A\}\subset \F_p.
\]
Since $M=\big\lfloor p^{\frac12-\frac{\varepsilon}{2}}\big\rfloor<p^{\frac{1}{2}}$ for all large $p$, the residues
$0^2,1^2,\dots,(M-1)^2$ are distinct modulo $p$, and every difference $a-a'$ is congruent to an
integer in $\{-(M-1),\dots,M-1\}$. Hence
\[
Q=\{0^2,1^2,\dots,(M-1)^2\}.
\]
Therefore
\[
\Delta_P(E)=D+Q=\bigcup_{k=0}^{M-1}(D+k^2).
\]
We next show the union is a single interval in $\mathbb{Z}$, and that no modular wrap-around occurs.
By definition of $M$ and $N$ we have $N\ge M$ for all large $p$, so consecutive squares satisfy
\[
(k+1)^2-k^2=2k+1\le 2M-1\le 2N-1.
\]
Therefore the intervals $D+k^2$ and $D+(k+1)^2$ overlap for each
$0\le k\le M-2$, and hence
$\bigcup\limits_{k=0}^{M-1}(D+k^2)$
is an interval with endpoints
\[
\min(D+0)=-(N-1),
\qquad
\max(D+(M-1)^2)=N-1+(M-1)^2.
\]
Consequently its cardinality equals
\[
\bigl(N-1+(M-1)^2\bigr)-\bigl(-(N-1)\bigr)+1=(M-1)^2+2N-1.
\]
Moreover,
\begin{align}\label{upper-bound-M-N}
    (M-1)^2+2N-1 \;\le\; M^2+2N \;\ll\; p^{1-\varepsilon}+p^{1-\frac{\varepsilon}{2}}=o(p),
\end{align}
so for all sufficiently large $p$ we have $M^2+2N<p$, and reduction modulo $p$ is injective on this interval. Therefore the same cardinality holds in $\F_p$, proving
\begin{align}\label{size-Delta-E}
    |\Delta_P(E)|=(M-1)^2+2N-1.
\end{align}
By construction,
\[
|E|=|A|\cdot |B|=MN \sim p^{(\frac12-\frac{\varepsilon}{2})+(1-\frac{\varepsilon}{2})}
= p^{\frac{3}{2}-\varepsilon},
\]
while \eqref{size-Delta-E} together with \eqref{upper-bound-M-N} implies that
\[
|\Delta_P(E)| \ll\; p^{1-\varepsilon}+p^{1-\frac{\varepsilon}{2}}=o(p).
\]
This completes the proof of the proposition.
\end{proof}

    One could generalize the above construction for $q$ of the form $p^\ell,$ where $p$ is an odd prime, but there is some additional bookkeeping. Let $\alpha$ be a generator for the multiplicative group of $\mathbb F_q.$ Because $\mathbb F_q$ is a finite simple field extension of $\mathbb F_p,$ the set $\{1, \alpha, \alpha^2, \dots, \alpha^{\ell-1}\}$ forms a polynomial basis for $\mathbb F_q$ considered as a vector space over $\mathbb F_p.$ We now follow a procedure similar to that in the proof of Proposition \ref{prop:prime-power-counterexample} in each dimension of the vector space, but in order to avoid the modular wrap-around, we would need to take smaller and smaller proportions of each copy of $\mathbb F_p$ for larger choices of $\ell$, and the method would fail if $\ell$ is sufficiently large with respect to $p$. To avoid this pitfall, we use a different approach to prove the next result.

\begin{proposition}\label{prop:prime-power-counterexample}
Let $q$ be a sufficiently large odd prime power and let $\varepsilon\in(0,\tfrac12)$ be a rational number.
Then there exists a set $E\subset \F_q^2$ such that
\[
|E| = q^{\frac{3}{2}-\varepsilon}
\qquad\text{and}\qquad
|\Delta_P(E)|=o(q).
\]
\end{proposition}

\begin{proof}
Let $q=p^{2m}$ with $p$ odd, where $m$ is a rational number chosen so that $(1-\varepsilon)2m\in\mathbb{N}$.
Such a choice is possible by taking $m$ to be a multiple of the denominator of $1-\varepsilon$.
Let
\[
H:=\F_{p^{m}}\subset \F_{p^{2m}}=\F_q,
\qquad |H|=p^m=q^{\frac{1}{2}}.
\]
Viewing $\F_q$ as a $2m$-dimensional vector space over $\F_p$, the set $H$ is an $\F_p$-subspace of dimension $m$.
Choose an $\F_p$-subspace $V\subset\F_q$ of dimension $(1-\varepsilon)2m$
containing $H$. Such a choice is possible since $m\le (1-\varepsilon)2m$ by assumption on $\varepsilon$.  
Then $|V| = p^{(1-\varepsilon)\cdot 2m} = q^{1-\varepsilon}$.

Define $E:=H\times V \subset \F_q^2$.
Thus 
\[
|E| = |H| \, |V| = q^{\frac{1}{2}}\cdot q^{1-\varepsilon} = q^{\frac{3}{2}-\varepsilon}. 
\]
We claim that $\Delta_P(E)\subset V$.
Indeed, take $\mathbf{x}=(h,u)$ and $\mathbf{y}=(h',u')$ in $E$, so $h,h'\in H$ and $u,u'\in V$.
Since $V$ is an additive subgroup of $\F_q$, we have $u-u'\in V$.
Also $h-h'\in H$, hence $(h-h')^2\in H\subset V$.
Therefore
\[
\|\mathbf{x}-\mathbf{y}\|_P
=(u-u')+(h-h')^2
\in V,
\]
which proves $\Delta_P(E)\subset V$.
Consequently
\[
|\Delta_P(E)|\le |V|=q^{1-\varepsilon}=o(q).
\]
The proof of the proposition is complete.
\end{proof}

\section{Proof of Theorem \ref{thm:main}}\label{sec:proof-main}

Define the distance counting function $\nu:\F_q\to\mathbb N$ by
\begin{equation*}
\nu(t):=\abs{\{(\mathbf{x}, \mathbf{y})\in E^2:\ \|\mathbf{x}-\mathbf{y}\|_{P}=t\}},
\qquad t\in\F_q.
\end{equation*}
Then $\sum\limits_{t\in \F_q}\nu(t)=|E|^2$. By the Cauchy--Schwarz inequality, one has
\[
|E|^4
=
\left(\sum_{t\in \F_q}\nu(t)\right)^2
\le
\abs{\{t\in \F_q:\nu(t)\ne 0\}}
\cdot
\sum_{t\in \F_q}\nu(t)^2
=
\abs{\Delta_P(E)}\sum_{t\in \F_q}\nu(t)^2.
\]
This implies 
\begin{equation}\label{eq111}
\abs{\Delta_P(E)}\ \ge\ \frac{n^4}{\sum\limits_{t\in\F_q}\nu(t)^2}.
\end{equation}
Therefore, in the rest, we focus on bounding $\sum\limits_{t\in \F_q}\nu(t)^2$ from above.

\textbf{The shear transformations.} We now introduce a pair of transformations and some of their properties. Recall that since $q$ is odd, $2$ is invertible in $\F_q$. Now, for any set $E\subseteq \mathbb F_q^2,$ define
\begin{equation}\label{eq:def-shears}
A:=\{(x_1,x_2+x_1^2):\mathbf{x}=(x_1,x_2)\in E\},
\qquad
B:=\{(y_1,y_2-y_1^2):\mathbf{y}=(y_1,y_2)\in E\}.
\end{equation}
Each of the maps $\mathbf{x}\mapsto(x_1,x_2+x_1^2)$ and $\mathbf{y}\mapsto(y_1,y_2-y_1^2)$ is a bijection on $\F_q^2$, so
\begin{equation*}
|A|=|B|=|E|.
\end{equation*}
First we relate the shear transformation to our counting function.
\begin{lemma}\label{lem:shear-identity}
For $\mathbf{x},\mathbf{y}\in E$, let $\mathbf{a}\in A$ and $\mathbf{b}\in B$ be the corresponding points under \eqref{eq:def-shears}. Then
\[
(x_2-y_2)+(x_1-y_1)^2=a_2-b_2-2a_1b_1.
\]
Consequently,
\[
\nu(t)=\abs{\{(\mathbf{a},\mathbf{b})\in A\times B:\ a_2-b_2-2a_1b_1=t\}}.
\]
\end{lemma}

\begin{proof}
Expanding directly,
\[
(x_2-y_2)+(x_1-y_1)^2
=
x_2-y_2+x_1^2-2x_1y_1+y_1^2
=
(x_2+x_1^2)-(y_2-y_1^2)-2x_1y_1
=
a_2-b_2-2a_1b_1. \qedhere
\]
\end{proof}

It is clear by definition that the shear transformation preserves vertical fibers. The following is a particularly useful consequence of this fact.

\begin{lemma}\label{lem:K-preserved}
For every $u\in \F_q$,
\[
\abs{A\cap(\{u\}\times \F_q)}=\abs{B\cap(\{u\}\times \F_q)}=\abs{E\cap(\{u\}\times \F_q)}.
\]
In particular, $\max\limits_{u\in \F_q}\abs{A\cap(\{u\}\times\F_q)}=\max\limits_{u\in \F_q}\abs{B\cap(\{u\}\times\F_q)}=K_E$.
\end{lemma}

\begin{proof}
For fixed $u$, both maps in \eqref{eq:def-shears} restrict to bijections on the vertical line $\{u\}\times\F_q$, so they preserve fiber sizes.
\end{proof}

The heart of the argument is the following second-moment bound.

\begin{proposition}\label{prop:second-moment}
With $n:=|E|$ and $K_E$ as in \eqref{eq:def-K},
\begin{equation}\label{eq:second-moment-bound}
\sum_{t\in \F_q}\nu(t)^2
\le
\frac{n^4}{q}+qn^2K_E.
\end{equation}
\end{proposition}

\begin{proof}
By the definition of the shear transformations, Lemma~\ref{lem:shear-identity} and the orthogonality relation \eqref{eq:orthogonality}, we write
\begin{equation}\label{eq:nu-fourier-expansion}
\nu(t)
=
\frac{1}{q}\sum_{s\in \F_q}\chi(-st)S(s),
\qquad\text{where}\quad
S(s):=
\sum_{\mathbf{a}\in A}\sum_{\mathbf{b}\in B}\chi\bigl(s(a_2-b_2+2a_1b_1)\bigr).
\end{equation}
The $s=0$ term contributes $\frac{n^2}{q}$. Setting $R_t:=\nu(t)-\frac{n^2}{q}$, we have $\sum\limits_{t\in \F_q} R_t=0$ and
\begin{equation}\label{eq:nu-split}
\sum_{t\in \F_q}\nu(t)^2
=
\frac{n^4}{q}+\sum_{t\in \F_q}R_t^2.
\end{equation}
Applying orthogonality to $\sum\limits_{t\in \F_q} R_t^2$ gives
\begin{equation}\label{eq:R-by-S}
\sum_{t\in \F_q}R_t^2
=
\frac{1}{q}\sum_{s\ne 0}|S(s)|^2.
\end{equation}
It remains to bound $\sum\limits_{s\ne 0}|S(s)|^2$.

For $x,y\in\F_q$, define vertical slices
\[
A_x:=\{u\in\F_q:(x,u)\in A\},
\qquad
B_y:=\{v\in\F_q:(y,v)\in B\},
\]
and for $s\ne 0$ set
\[
f_s(x):=\sum_{u\in A_x}\chi(su),
\qquad
g_s(y):=\sum_{v\in B_y}\chi(-sv).
\]
Writing $\mathbf{a}=(x,u)$ and $\mathbf{b}=(y,v)$, we rewrite \eqref{eq:nu-fourier-expansion} as
\[
S(s)=\sum_{(x,u)\in A}\sum_{(y,v)\in B}\chi(s(u-v+2xy))=\sum_{x\in\mathbb F_q}f_s(x)\sum_{y\in\mathbb F_q}g_s(y)\chi(-2sxy)=\sum_{x\in\F_q}f_s(x)\,\widehat{g_s}(-2sx).
\]
Define $U_s:=\sum\limits_{x\in\F_q}|f_s(x)|^2$ and $V_s:=\sum\limits_{y\in\F_q}|g_s(y)|^2$. Since $x\mapsto -2sx$ is a bijection of $\F_q$ for $s\ne 0$, the Cauchy--Schwarz inequality and Plancherel \eqref{lem:plancherel} give
\begin{equation*}
|S(s)|^2
\le
U_s\sum_{\xi\in\F_q}|\widehat{g_s}(\xi)|^2
=qU_sV_s.
\end{equation*}
Hence, by the Cauchy--Schwarz inequality in $s$,
\begin{equation}\label{eq:sumS-by-UV}
\sum_{s\ne 0}|S(s)|^2
\le
q\sum_{s\ne 0}U_sV_s
\le
q\left(\sum_{s\in\F_q}U_s^2\right)^{\frac{1}{2}}
\left(\sum_{s\in\F_q}V_s^2\right)^{\frac{1}{2}}.
\end{equation}

We now estimate $\sum\limits_{s\in \F_q} U_s^2$. Expanding and applying orthogonality \eqref{eq:orthogonality},
\[
\sum_{s\in\F_q}U_s^2
=
\sum_{s\in\mathbb F_q}\sum_{u,w\in A_x}\chi(s(u-w))\sum_{u',w'\in A_x}\chi(s(u'-w'))
=
q\sum_{x,x'\in\F_q}E_+(A_x,A_{x'}).
\]
By Lemma~\ref{lem:energy-trivial}, $E_+(A_x,A_{x'})\le|A_x|^{\frac{3}{2}}|A_{x'}|^{\frac{3}{2}}$, so
\[
\sum_{s\in\F_q}U_s^2
\le
q\left(\sum_{x\in\F_q}|A_x|^{\frac{3}{2}}\right)^2.
\]
We pause here to note that this is a fairly brusque way to estimate additive energy. We address this by providing a refined argument in Subsection \ref{sec:fiber-energy}. However, the estimate we use here will suffice for this argument. Also, to ease exposition, here and below, we write $K$ in place of $K_E,$ because context is clear. Since $|A_x|\le K$ for all $x$ (Lemma~\ref{lem:K-preserved}) and $\sum\limits_{x\in\F_q}|A_x|=n$,
\[
\sum_{x\in\F_q}|A_x|^{\frac{3}{2}}\le K^{\frac{1}{2}}\sum_{x\in\F_q}|A_x|=K^{\frac{1}{2}}n,
\]
which gives
\begin{equation}\label{eq:Us-bound}
\sum_{s\in\F_q}U_s^2\le qKn^2.
\end{equation}
The same argument applied to the slices of $B$ yields
\begin{equation}\label{eq:Vs-bound}
\sum_{s\in\F_q}V_s^2\le qKn^2.
\end{equation}
Combining \eqref{eq:sumS-by-UV}, \eqref{eq:Us-bound}, and \eqref{eq:Vs-bound},
\[
\sum_{s\ne 0}|S(s)|^2\le q^2Kn^2.
\]
Together with \eqref{eq:nu-split} and \eqref{eq:R-by-S}, this gives
\[
\sum_{t\in \F_q}\nu(t)^2
\le
\frac{n^4}{q}+qn^2K,
\]
completing the proof.
\end{proof}
With Proposition~\ref{prop:second-moment} in hand, it follows from  (\ref{eq111}) that
\[
|\Delta_P(E)|
\ge
\frac{|E|^4}{\dfrac{|E|^4}{q}+q|E|^2K}
=
\frac{q}{1+\dfrac{q^2K}{|E|^2}}.
\]
This completes the proof of Theorem~\ref{thm:main}.

\begin{proof}[Proof of Corollary~\ref{cor:alpha}]
By Theorem~\ref{thm:main}, it suffices to ensure $|E|^2\gg q^2K$. Under $K\ll|E|^\alpha$, this reduces to $|E|^{2-\alpha}\gg q^2$, that is, $|E|\gg q^{\frac{2}{2-\alpha}}$.
\end{proof}

\section{Some extensions of Theorem \ref{thm:main}}\label{extension-section}

Here, we state and prove two logical extensions of Theorem \ref{thm:main}. The first is a slight refinement of the original statement, and the second is a generalization to handle parabolic distances between two sets.

\subsection{A fiber-energy refinement}\label{sec:fiber-energy}
As mentioned in the proof Theorem \ref{thm:main}, one of the key obstructions was additive energy between various vertical lines. Rather than state this explicitly in the main result, we used a courser (yet effective) measure: the maximum number of points on any vertical line. Here, we state and prove a variant of Theorem \ref{thm:main} that does not make this simplification. So overall, the result is stronger, but the hypotheses also harder to verify directly.

Let $\nu(t)$ be the distance counting function introduced in Section~\ref{sec:proof-main}. For each $u\in \F_q$, define the vertical fiber
\begin{equation}\label{eq:def-Eu}
E_u:=\{y\in\F_q:(u,y)\in E\}.
\end{equation}
We also define the \emph{fiber alignment energy} of $E$ by
\begin{equation}\label{eq:def-fiber-energy}
\mathcal{E}_{\mathrm{fib}}(E)
:=
\sum_{u,u'\in\F_q}E_+(E_u,E_{u'}).
\end{equation}
The proof of Proposition~\ref{prop:second-moment} yields the following refinement once one retains the additive energies of the fibers instead of estimating them through the single parameter $K$.

\begin{theorem}\label{thm:fiber-energy}
Let $E\subset \F_q^2$ and write $n:=|E|$. Then
\[
\sum_{t\in\F_q}\nu(t)^2
\le
\frac{n^4}{q}+q\,\mathcal{E}_{\mathrm{fib}}(E).
\]
Consequently,
\[
|\Delta_P(E)|
\ge
\frac{n^4}{\dfrac{n^4}{q}+q\,\mathcal{E}_{\mathrm{fib}}(E)}
=
\frac{q}{1+\dfrac{q^2\mathcal{E}_{\mathrm{fib}}(E)}{n^4}}.
\]
\end{theorem}

\begin{proof}
We retain the notation introduced in the proof of Proposition~\ref{prop:second-moment}. In particular, the derivation of \eqref{eq:nu-split}, \eqref{eq:R-by-S}, and \eqref{eq:sumS-by-UV} is unchanged, so it remains only to compute $\sum\limits_{s\in\F_q}U_s^2$ and $\sum\limits_{s\in\F_q}V_s^2$ more precisely.

For each $u\in\F_q$, the vertical slices of the shear sets $A$ and $B$ satisfy
\[
A_u=E_u+u^2,
\qquad
B_u=E_u-u^2.
\]
Since additive energy is invariant under translating either input set, it follows that
\[
E_+(A_u,A_{u'})=E_+(B_u,B_{u'})=E_+(E_u,E_{u'})
\qquad\text{for all }u,u'\in\F_q.
\]
Now the same computation used in the proof of Proposition~\ref{prop:second-moment} gives
\[
\sum_{s\in\F_q}U_s^2
=
q\sum_{u,u'\in\F_q}E_+(A_u,A_{u'})
=
q\,\mathcal{E}_{\mathrm{fib}}(E),
\]
and similarly,
\[
\sum_{s\in\F_q}V_s^2
=
q\sum_{u,u'\in\F_q}E_+(B_u,B_{u'})
=
q\,\mathcal{E}_{\mathrm{fib}}(E).
\]
Therefore \eqref{eq:sumS-by-UV} yields
\[
\sum_{s\ne 0}|S(s)|^2
\le
q\left(\sum_{s\in\F_q}U_s^2\right)^{\frac{1}{2}}
\left(\sum_{s\in\F_q}V_s^2\right)^{\frac{1}{2}}
=
q^2\,\mathcal{E}_{\mathrm{fib}}(E).
\]
Substituting this bound into \eqref{eq:nu-split} and \eqref{eq:R-by-S}, we obtain
\[
\sum_{t\in\F_q}\nu(t)^2
\le
\frac{n^4}{q}+q\,\mathcal{E}_{\mathrm{fib}}(E).
\]
The lower bound for $|\Delta_P(E)|$ now follows from \eqref{eq111} exactly as before.
\end{proof}

\begin{remark}
Theorem~\ref{thm:main} is an immediate consequence of Theorem~\ref{thm:fiber-energy}. Indeed, Lemma~\ref{lem:energy-trivial} gives
\[
\mathcal{E}_{\mathrm{fib}}(E)
\le
\sum_{u,u'\in\F_q}K|E_u| |E_{u'}|
=
K\left(\sum_{u\in\F_q}|E_u|\right)^2
=
Kn^2.
\]
\end{remark}

It is not uncommon to see additive energy play an active role in distance problems. In \cite{fraser2}, Fraser produced a host of results following from Fourier properties. Following this, in \cite{CGKPTZ25}, the authors used additive energy to measure the relevant Fourier properties, which in turn guaranteed stronger results for distance problems. One should note that the counterexamples in Section \ref{sec:prime-counterexample} consist of families of vertical lines with maximal fiber-energy.

\subsection{A bipartite extension}\label{sec:bipartite}

One can modify the analyses above to handle parabolic distances between two sets. For two sets $E,F\subset \F_q^2$, define the \emph{bipartite parabolic distance set} by
\[
\Delta_P(E,F):=\{\|\mathbf{x}-\mathbf{y}\|_P:\ \mathbf{x}\in E,\ \mathbf{y}\in F\}.
\]
We recall
\[
K_E:=\max_{u\in \F_q}\abs{E\cap (\{u\}\times \F_q)},
\qquad
K_F:=\max_{u\in \F_q}\abs{F\cap (\{u\}\times \F_q)}.
\]

\begin{theorem}\label{thm:bipartite}
Let $E,F\subset \F_q^2$. Then
\[
\sum_{t\in\F_q}\nu_{E,F}(t)^2
\le
\frac{|E|^2|F|^2}{q}
+
q\,|E|\,|F|\,\sqrt{K_EK_F},
\]
where
\[
\nu_{E,F}(t):=
\abs{\{(\mathbf{x},\mathbf{y})\in E\times F:\ \|\mathbf{x}-\mathbf{y}\|_P=t\}}.
\]
Consequently,
\[
\abs{\Delta_P(E,F)}
\ge
\frac{|E|^2|F|^2}
{\dfrac{|E|^2|F|^2}{q}+q\,|E|\,|F|\,\sqrt{K_EK_F}}
=
\frac{q}{1+\dfrac{q^2\sqrt{K_EK_F}}{|E|\,|F|}}.
\]
In particular, if $|E|\,|F|\gg q^2\sqrt{K_EK_F}$, then $\abs{\Delta_P(E,F)}\gg q$.
\end{theorem}

\begin{proof}
We follow the proof of Proposition~\ref{prop:second-moment}, replacing the two copies of $E$ by the pair $E,F$.

For $t\in \F_q$, define
\[
\nu_{E,F}(t):=
\abs{\{(\mathbf{x},\mathbf{y})\in E\times F:\ \|\mathbf{x}-\mathbf{y}\|_P=t\}}.
\]
Then
\[
\sum_{t\in\F_q}\nu_{E,F}(t)=|E|\,|F|.
\]
Hence, by Cauchy--Schwarz,
\begin{equation}\label{eq:bipartite-CS}
|E|^2|F|^2
=
\left(\sum_{t\in\F_q}\nu_{E,F}(t)\right)^2
\le
\abs{\Delta_P(E,F)}
\sum_{t\in\F_q}\nu_{E,F}(t)^2,
\end{equation}
so it remains to bound the second moment $\sum\limits_{t\in\F_q}\nu_{E,F}(t)^2$.

As in \eqref{eq:def-shears}, define
\[
A:=\{(x_1,x_2+x_1^2):\ \mathbf{x}=(x_1,x_2)\in E\},
\qquad
B:=\{(y_1,y_2-y_1^2):\ \mathbf{y}=(y_1,y_2)\in F\}.
\]
Then $|A|=|E|$ and $|B|=|F|$, and for $\mathbf{x}\in E$, $\mathbf{y}\in F$, with corresponding points
$\mathbf{a}\in A$, $\mathbf{b}\in B$, one has
\[
\|\mathbf{x}-\mathbf{y}\|_P
=
(x_2-y_2)+(x_1-y_1)^2
=
a_2-b_2-2a_1b_1.
\]
Therefore
\[
\nu_{E,F}(t)
=
\abs{\{(\mathbf{a},\mathbf{b})\in A\times B:\ a_2-b_2-2a_1b_1=t\}}.
\]

Using the orthogonality relation \eqref{eq:orthogonality}, we obtain
\[
\nu_{E,F}(t)
=
\frac{1}{q}\sum_{s\in \F_q}\chi(-st)S(s),
\]
where
\[
S(s):=
\sum_{\mathbf{a}\in A}\sum_{\mathbf{b}\in B}
\chi\bigl(s(a_2-b_2-2a_1b_1)\bigr).
\]
The $s=0$ term contributes $\frac{|E| |F|}{q}$. Setting
\[
R_t:=\nu_{E,F}(t)-\frac{|E| |F|}{q},
\]
the same computation as in \eqref{eq:nu-split} and \eqref{eq:R-by-S} gives
\[
\sum_{t\in\F_q}\nu_{E,F}(t)^2
=
\frac{|E|^2|F|^2}{q}
+
\frac{1}{q}\sum_{s\ne 0}|S(s)|^2.
\]
Thus it remains to estimate $\sum\limits_{s\ne 0}|S(s)|^2$.

For $x,y\in\F_q$, define the vertical slices
\[
A_x:=\{u\in\F_q:(x,u)\in A\},
\qquad
B_y:=\{v\in\F_q:(y,v)\in B\},
\]
and for $s\ne 0$ set
\[
f_s(x):=\sum_{u\in A_x}\chi(su),
\qquad
g_s(y):=\sum_{v\in B_y}\chi(-sv).
\]
Then
\[
S(s)=\sum_{x\in\F_q}f_s(x)\,\widehat{g_s}(-2sx).
\]
Define
\[
U_s:=\sum_{x\in\F_q}|f_s(x)|^2,
\qquad
V_s:=\sum_{y\in\F_q}|g_s(y)|^2.
\]
Exactly as in the proof of Proposition~\ref{prop:second-moment}, the Cauchy--Schwarz inequality and Plancherel \eqref{lem:plancherel} yield
\[
|S(s)|^2\le q\,U_sV_s
\qquad (s\ne 0),
\]
and hence
\begin{equation}\label{eq:bipartite-sumS}
\sum_{s\ne 0}|S(s)|^2
\le
q\left(\sum_{s\in\F_q}U_s^2\right)^{\frac{1}{2}}
\left(\sum_{s\in\F_q}V_s^2\right)^{\frac{1}{2}}.
\end{equation}

We now estimate the two factors separately. By the same expansion used in the proof of
Proposition~\ref{prop:second-moment},
\[
\sum_{s\in\F_q}U_s^2
=
q\sum_{x,x'\in\F_q}E_+(A_x,A_{x'}).
\]
Applying Lemma~\ref{lem:energy-trivial},
\[
E_+(A_x,A_{x'})
\le
|A_x|^{\frac32}|A_{x'}|^{\frac32},
\]
so
\[
\sum_{s\in\F_q}U_s^2
\le
q\left(\sum_{x\in\F_q}|A_x|^{\frac32}\right)^2.
\]
Since the shear preserves vertical fibers, we have
\[
\max_{x\in\F_q}|A_x|=K_E
\qquad\text{and}\qquad
\sum_{x\in\F_q}|A_x|=|E|.
\]
Therefore
\[
\sum_{x\in\F_q}|A_x|^{\frac32}
\le
K_E^{\frac{1}{2}}\sum_{x\in\F_q}|A_x|
=
K_E^{\frac{1}{2}}|E|,
\]
and hence
\[
\sum_{s\in\F_q}U_s^2\le qK_E|E|^2.
\]
By the same argument applied to $B$,
\[
\sum_{s\in\F_q}V_s^2\le qK_F|F|^2.
\]
Substituting these bounds into \eqref{eq:bipartite-sumS}, we find
\[
\sum_{s\ne 0}|S(s)|^2
\le
q^2|E|\,|F|\,\sqrt{K_EK_F}.
\]
Consequently,
\[
\sum_{t\in\F_q}\nu_{E,F}(t)^2
\le
\frac{|E|^2|F|^2}{q}
+
q\,|E|\,|F|\,\sqrt{K_EK_F}.
\]
Finally, combining this with \eqref{eq:bipartite-CS} gives
\[
\abs{\Delta_P(E,F)}
\ge
\frac{|E|^2|F|^2}
{\dfrac{|E|^2|F|^2}{q}+q\,|E|\,|F|\,\sqrt{K_EK_F}}
=
\frac{q}{1+\dfrac{q^2\sqrt{K_EK_F}}{|E|\,|F|}},
\]
as claimed.
\end{proof}

\begin{remark}
Retaining the fiber energies, exactly as in Section~\ref{sec:fiber-energy}, one obtains the sharper mixed bound
\[
\sum_{t\in\F_q}\nu_{E,F}(t)^2
\le
\frac{|E|^2|F|^2}{q}
+
q\bigl(\mathcal{E}_{\mathrm{fib}}(E)\mathcal{E}_{\mathrm{fib}}(F)\bigr)^{\frac{1}{2}}.
\]
Theorem~\ref{thm:bipartite} follows from this together with the estimate
\[
\mathcal{E}_{\mathrm{fib}}(E)\le K_E|E|^2,
\qquad
\mathcal{E}_{\mathrm{fib}}(F)\le K_F|F|^2.
\]
\end{remark}

\section{On the sharpness of Theorem \ref{thm:main}}\label{section6}

We now introduce variants of the results from Section \ref{sec:prime-counterexample} modified to suit the setting of Theorem \ref{thm:main}.

\begin{proposition}
Fix \(0<\varepsilon<1\). Then for all sufficiently large odd primes \(p\) there exists
\(E\subset \F_p^2\) such that, with
\[
K_E:=\max_{u\in\F_p}\bigl|E\cap (\{u\}\times \F_p)\bigr|,
\]
one has
\[
|E|\sim \frac{p\,K_E^{\frac{1}{2}}}{p^\varepsilon}
\qquad\text{and}\qquad
|\Delta_P(E)|=o(p).
\]
More precisely, one may arrange
\[
K_E\sim p^{1-\varepsilon},\qquad
|E|\sim p^{\frac32-\frac{3\varepsilon}{2}},\qquad
|\Delta_P(E)|\sim p^{1-\varepsilon}.
\]
\end{proposition}

\begin{proof}
Set
\[
M:=\Bigl\lfloor p^{\frac12-\frac{\varepsilon}{2}}\Bigr\rfloor,
\qquad
N:=\lfloor p^{1-\varepsilon}\rfloor,
\]
and define
\[
A:=\{0,1,\dots,M-1\}\subset \F_p,
\qquad
B:=\{0,1,\dots,N-1\}\subset \F_p,
\qquad
E:=A\times B\subset \F_p^2.
\]
Since each vertical fiber above \(a\in A\) has size \(N\), we have $K_E=N\sim p^{1-\varepsilon}$.
Hence
\[
|E|=|A| |B|=MN
\sim p^{\frac12-\frac{\varepsilon}{2}}\,p^{1-\varepsilon}
= p^{\frac32-\frac{3\varepsilon}{2}}
\sim \frac{p\sqrt{K_E}}{p^\varepsilon}.
\]
Now write \(\mathbf{x}=(a,b)\), \(\mathbf{y}=(a',b')\) with \(a,a'\in A\), \(b,b'\in B\). Then
\[
\|\mathbf{x}-\mathbf{y}\|_P=(b-b')+(a-a')^2.
\]
Let
\[
D:=B-B=\{-(N-1),\dots,N-1\},
\qquad
Q:=\{(a-a')^2:\ a,a'\in A\}.
\]
Since \(M<p^{\frac{1}{2}}\) for all large \(p\), the residues
$0^2,1^2,\dots,(M-1)^2$
are distinct mod \(p\), and every difference \(a-a'\) is represented by an integer in
\(\{-(M-1),\dots,M-1\}\). Thus
\[
Q=\{0^2,1^2,\dots,(M-1)^2\}.
\]
Therefore
\[
\Delta_P(E)=D+Q=\bigcup_{k=0}^{M-1}(D+k^2).
\]
Because \(N\ge M\), consecutive translates overlap:
\[
(k+1)^2-k^2=2k+1\le 2M-1\le 2N-1.
\]
So the union is a single interval of length
$(M-1)^2+2N-1$.

Moreover,
\[
(M-1)^2+2N-1\le M^2+2N \ll p^{1-\varepsilon}=o(p),
\]
so there is no modular wrap-around for large \(p\). Hence
\[
|\Delta_P(E)|=(M-1)^2+2N-1\sim p^{1-\varepsilon}=o(p).
\]
This completes the proof.
\end{proof}

\begin{proposition}\label{prop:prime-power-exact-sharpness}
Fix a rational number $0 < \varepsilon < 1$. Let $k \ge 2$ be an integer such that $\frac{1}{1-\varepsilon} \le k \le \frac{2}{1-\varepsilon}$. 
Let $q=p^{km}$ be an odd prime power where $m$ is chosen so that 
$\frac{1}{2}(1-\varepsilon)km$ is integer.
Then there exists a set $E\subset \F_q^2$ such that, with
\[
K_E:=\max_{x\in\F_q}\bigl|E\cap (\{x\}\times \F_q)\bigr|,
\]
one has
\[
|E| = \frac{q\,K_E^{\frac{1}{2}}}{q^\varepsilon}
\qquad\text{and}\qquad
|\Delta_P(E)|=o(q).
\]
More precisely, one may arrange
\[
K_E = q^{1-\varepsilon},\qquad
|E| = q^{\frac32-\frac{3\varepsilon}{2}},\qquad
|\Delta_P(E)| = q^{1-\varepsilon}.
\]
\end{proposition}

\begin{proof}
By the choice of $k$, we have $1 \le k(1-\varepsilon) \le 2$. Let 
\[
H:=\F_{p^{m}}\subset \F_{p^{km}}=\F_q,
\qquad |H|=p^m=q^{\frac{1}{k}}.
\]
Viewing $\F_q$ as a $km$-dimensional vector space over $\F_p$, the subfield $H$ is an $\F_p$-subspace of dimension $m$. Since
$1 \le k(1-\varepsilon) \le 2$,
it follows that $\frac{1}{2}(1-\varepsilon)km \le m \le (1-\varepsilon)km$.
Therefore, we can choose an $\F_p$-subspace $V\subset \F_q$ of dimension $(1-\varepsilon)km$ containing $H$, and an $\F_p$-subspace $U\subset H$ of dimension $\frac{1}{2}(1-\varepsilon)km$. In particular,
\[
|V| = p^{(1-\varepsilon)km} = q^{1-\varepsilon},
\qquad
|U| = p^{\frac{1}{2}(1-\varepsilon)km} = q^{\frac{1-\varepsilon}{2}}.
\]

Define $E:=U\times V \subset \F_q^2$.
For this set, the vertical fibers are non-empty only over $x\in U$, and each such fiber is precisely $\{x\}\times V$. Thus,
\[
K_E = \max_{x\in\F_q}\bigl|E\cap (\{x\}\times \F_q)\bigr| = |V| = q^{1-\varepsilon}.
\]
The cardinality of $E$ is exactly
\[
|E| = |U| \, |V| = q^{\frac{1-\varepsilon}{2}}\cdot q^{1-\varepsilon} = q^{\frac{3}{2}-\frac{3\varepsilon}{2}}. 
\]
Notice that this precisely satisfies the target relation:
\[
\frac{q\,K_E^{\frac{1}{2}}}{q^\varepsilon} = \frac{q \cdot (q^{1-\varepsilon})^{\frac{1}{2}}}{q^\varepsilon} = q^{1 + \frac{1-\varepsilon}{2} - \varepsilon} = q^{\frac{3}{2}-\frac{3\varepsilon}{2}} = |E|.
\]
We now bound the size of the parabolic distance set $\Delta_P(E)$.
Let $\mathbf{x}=(u,v)$ and $\mathbf{y}=(u',v')$ be elements of $E$, so $u,u'\in U$ and $v,v'\in V$.
By definition, 
\[
\|\mathbf{x}-\mathbf{y}\|_P = (v-v')+(u-u')^2.
\]
Since $U\subset H\subset V$ and both $U$ and $V$ are additive subgroups of $\F_q$, we have $u-u'\in U\subset H$ and $v-v'\in V$. As $H$ is a subfield, it is closed under multiplication, hence $(u-u')^2\in H\subset V$. Consequently, both $v-v'$ and $(u-u')^2$ lie in $V$.
Therefore, the parabolic distance $\|\mathbf{x}-\mathbf{y}\|_P \in V$, which proves that $\Delta_P(E)\subset V$.

Conversely, since $0 \in U$, fixing $u=u'=0$ shows that the differences $v-v'$ cover all of $V$. Thus $\Delta_P(E) = V$.
Consequently,
\[
|\Delta_P(E)| = |V| = q^{1-\varepsilon} = o(q).
\]
The proof of the proposition is complete.
\end{proof}

Of course, in the setting of Theorem \ref{thm:main}, if the set $E$ under consideration has most of its mass concentrated on vertical lines with fewer than $K_E$ points, one could get better results for subsets of $E.$ That is, if there is a subset $F\subset E$ with $|F|\sim|E|,$ but $K_F$ is significantly smaller than $K_E,$ then applying Theorem \ref{thm:main} to the subset $F$ would give lower bounds on $|\Delta(F)|$ which would, in turn, imply lower bounds on $|\Delta(E)|.$ However, we do not pursue this further, as we have shown that the main obstructions appear to occur in cases where $E$ is relatively well-distributed across vertical slices, so even such a pruning scheme could not help in the extreme cases.

{\bf Acknowledgments.} 
Dung The Tran was supported by the research project QG.25.02 of Vietnam National University, Hanoi. He also would like to thank Vietnam Institute for Advanced Study in Mathematics (VIASM) for the hospitality and for the excellent working conditions.

\end{document}